\documentclass[12pt,leqno]{article}
\usepackage{amssymb}
\usepackage{amsmath}
\usepackage{tabularx}
\usepackage{enumitem}
\usepackage{theorem}
\newtheorem{theorem}[equation]{Theorem}

\newtheorem{definition}[equation]{Defnition
}
\newtheorem{lemma}[equation]{Lemma}

\newtheorem{proposition}[equation]{Proposition}

{\theorembodyfont{\rmfamily}

\newtheorem{remark}[equation]{Remark}

}

\newenvironment{proof}[1][Proof]{\noindent\textbf{#1.} }{\

\rule{0.5em}{0.5em}\medskip}

\newcommand{\R}{\mathbb R}

\newcommand{\C}{\mathbb C}
\newcommand{\Z}{\mathbb Z}

\renewcommand{\k}{\mathfrak k}

\newcommand\inv{^{-1}}

\newcommand{\g}{\mathfrak g}

\newcommand{\sgn}{\text{sgn}}
\newcommand{\diag}{\text{diag}}

\renewcommand{\int}{\text{int}}

\renewcommand{\sec}[1]{\section{#1}
\renewcommand{\theequation}{\thesection.\arabic{equation}}
  \setcounter{equation}{0}}
\newcommand{\subsec}[1]{\subsection{#1}
\renewcommand{\theequation}{\thesection.\arabic{equation}}
\setcounter{equation}{0}}

\begin{document}
\title{Deforming representations of $SL(2,\R)$}
\author{Jeffrey Adams}
\maketitle{}

\sec{Introduction}

Consider the  spherical principal series representations $\pi(\nu)$ of $G=SL(2,\R)$.
Viewing these as Harish-Chandra modules, this is a family of
representations parametrized by $\nu\in\C$, 
realized on a fixed vector space $V$ with basis
$\{v_j\mid j\in2\Z\}$, with operators varying continuously in $\nu$.  If $\nu$
is not an odd integer then $\pi(\nu)$ is irreducible.  If $\nu=2n+1$
then $\pi(\nu)$ is indecomposable: the $n$-dimensional irreducible representation
is the unique irreducible submodule or quotient of
$\pi(\nu)$ (depending on the sign of $n$), and it has two infinite
dimensional quotients or submodules, accordingly.

It is interesting to consider whether it is possible to find a new family 
of representations, denoted $\Pi(\nu)$, varying continuously in $\nu$, so that
$\Pi(\nu)$ has the same composition factors as $\pi(\nu)$ for all
$\nu$  but 
such  that $\Pi(\nu)$ is {\it completely reducible} for all $\nu$. 

This turns out to be possible. Here is a precise statement, also incorporating 
the non-spherical principal series.  For $\epsilon=\pm1$ and
$\nu\in\C$ let $\chi=\chi(\epsilon,\nu)$ be the character of
$\R^\times$ given by $\chi(e^t)=e^{t\nu}$ and $\chi(-1)=\epsilon$.
Let $B$ be a Borel subgroup of $G$. Then the quotient of $B$ by its
nilradical is isomorphic to $\R^\times$.  Write $\chi(\epsilon,\nu)$ for
the resulting character of $B$. 

Let $\g$ be the complexified Lie algebra of $G$, and let $K\simeq S^1$ be a maximal compact subgroup of $G$. 
We define  $\pi(\epsilon,\nu)$ to be the underlying $(\g,K)$-module
of the principal series representation of $SL(2,\R)$, which is obtained by 
normalized induction from the character $\chi(\epsilon,\nu)$ of $B$.

It is well known (for example see \cite[Section 1.3]{vogan_green}) that
$\pi(\epsilon,\nu)$ is reducible if and only if $\nu\in\Z$ and
$(-1)^{\nu+1}=\epsilon$. Furthermore if $\pi(\epsilon,\nu)$ is
reducible, it is  completely reducible only if
$(\epsilon,\nu)=(-1,0)$. 
See Lemma \ref{l:reducibility}. 

\begin{theorem}
\label{t:main}
Fix $\epsilon=\pm1$. Then there is a family of $(\g,K)$-modules, denoted
$(\Pi(\epsilon,\nu),V_\epsilon)$, parametrized by $\nu\in\C$, realized on a fixed vector space $V_\epsilon$,
such that
\begin{enumerate}
\item For $X\in \g$ the operators $\Pi(\epsilon,\nu)(X)$ vary continuously in $\nu$;
\item For all $\nu$ $\Pi(\epsilon,\nu)$ has the same composition factors as $\pi(\epsilon,\nu)$;
\item For all $\nu$ $\Pi(\epsilon,\nu)$ is completely reducible.
\end{enumerate}
\end{theorem}
Explicit formulas for $\Pi(\epsilon,\nu)$ are given in Proposition \ref{p:main}.

Although this result might seem surprising, in fact it goes back to
Naimark's 1964 book on the representation theory of the Lorentz group
\cite{naimark}.  Naimark constructs the irreducible
representations of $SL(2,\C)$ by the reverse process. First he
algebraically constructs a family of representations of the Lie
algebra, which are completely reducible for all values of the
parameters. Then, by modifying the operators appropriately, and taking
care with the poles and zeros, he deforms this to give the usual 
families of principal series representations, which, when reducible, are 
usually indecomposable.

This paper arose from an attempt elucidate the construction of
\cite{naimark}. It would be interesting to see to what extent 
this generalizes to other groups. 

A notable feature of the families of operators $\Pi(\epsilon,\nu)$ is
that they are not differentiable in $\nu$ (see Proposition
\ref{p:main}). Analytic families of $(\g,K)$-modules are studied in
\cite{noort}, and  algebraic families of both representations and algebras are considered in 
\cite{bernstein_higson_subag}.
As pointed out by Eyal Subag, a family satisfying the conditions of the Theorem cannot
be algebraic, and hence does not appear in loc. cit.  See Remark
\ref{r:not}.

There is a close relationship with invariant Hermitian forms.  If
$\pi(\epsilon,\nu)$ is indecomposable then an invariant Hermitian form
on it cannot be nondegenerate, or the representation would be completely
reducible. However the family, $\Pi(\epsilon,\nu)$ being completely
reducible for all $\nu$ can, and in fact does, support a continuous
family of nondegenerate invariant Hermitian forms with nice properties. See Section
\ref{s:hermitian}.

\sec{A finite dimensional example}

We begin with a simple example which illustrates the main ideas.
Consider the group $G=\langle \R^+,\delta\rangle$ with relations $\delta^2=1$ and $\delta x\delta\inv=x\inv$. 
(This is a subgroup of index $2$ of the indefinite orthogonal group $O(1,1)$). 
For $\nu\in\C$ let $\pi(\nu)$ be the representation of $G$ induced from the character $x\rightarrow x^\nu$ of $\R^+$. 
Then $\pi(\nu)$ is two-dimensional; it is  irreducible if $\nu\ne 0$, whereas $\pi(0)$ is the 
direct sum of the two characters of $G/\R^+$. 
With a natural choice of coordinates we can write the operators as
$$
\begin{aligned}
\pi(\nu)(x)&=
\begin{pmatrix}
  x^\nu&0\\0&x^{-\nu}
\end{pmatrix}\\
\pi(\nu)(\delta)&=
\begin{pmatrix}
  0&1\\1&0
\end{pmatrix}.
\end{aligned}
$$
Clearly $\lim_{\nu\rightarrow 0}(\pi(\nu))=\pi(0)$, and this
representation is completely reducible.

The action of $\R^+$ can be written
$$
\pi(\nu)(e^t)=\exp(t X(\nu))
$$
where $X(\nu)=\diag(\nu,-\nu)$. 
Set
$Y(\nu)=
\begin{pmatrix}
  \nu&1\\0&-\nu
\end{pmatrix}$;
the key point is that $Y(\nu)$ is conjugate to $X(\nu)$ if (and only
if) $\nu\ne 0$.  Let $S(\nu)=
\begin{pmatrix}
  1&1\\0&-2\nu
\end{pmatrix}$, so that
$$
S(\nu)X(\nu)S(\nu)\inv=Y(\nu)\quad (\nu\ne 0).
$$
For $\nu\ne 0$ conjugate  $\pi(\nu)$ by $S(\nu)$  to define:
$$
\begin{aligned}
\Pi(\nu)(e^t)&=
S(\nu)\Pi(\nu)(e^t)S(\nu)\inv\\
&=\exp(t
\begin{pmatrix}
  \nu&1\\0&-\nu
\end{pmatrix})
=
\begin{pmatrix}
  e^{t\nu}&\frac1{2\nu}(e^{t\nu}-e^{-t\nu})\\0&e^{-t\nu}  
\end{pmatrix}
\\
\Pi(\nu)(\delta)&=
S(\nu)\Pi(\nu)(\delta)S(\nu)\inv=\begin{pmatrix}
  1&0\\-2\nu&-1
\end{pmatrix}
\end{aligned}
$$
By construction $\Pi(\nu)\simeq \pi(\nu)$ if $\nu\ne 0$.
However when we take the limit we see
$$
\begin{aligned}
\lim_{\nu\rightarrow 0}\Pi(\nu)(e^t)=
\begin{pmatrix}
1&t\\0&1
\end{pmatrix}\\
\lim_{\nu\rightarrow 0}\Pi(\nu)(\delta)=
\begin{pmatrix}
  1&0\\0&-1
\end{pmatrix}
\end{aligned}
$$
This has the same composition factors as $\pi(0)$, but the restriction to $\R^+$ is indecomposable.

\sec{Principal Series Representations of $SL(2,\R)$}

We explicitly construct the $(\g,K)$-modules of the principal series of $SL(2,\R)$. 
This is well known. We follow \cite{vogan_green}.

Let $G=SL(2,\R)$, $\g_0=\mathfrak s\mathfrak l(2,\R)$, $\g=\mathfrak s\mathfrak
l(2,\C)$, $K=SO(2)$, the usual maximal compact subgroup of
$SL(2,R)$.
Choose a basis of $\g$:
$$
E
=\frac12
\begin{pmatrix}
  1&i\\i&-1
\end{pmatrix},\quad
F=\frac12
\begin{pmatrix}
1&-i\\-i&-1
\end{pmatrix},\quad
H=
\begin{pmatrix}
0&-i\\i&0  
\end{pmatrix}
$$
These satisfy
\begin{equation}
\label{e:defining}
[H,E]=2E,\, [H,F]=-2F\,, [E,F]=H
\end{equation}
Then  $\g=\mathfrak k\oplus \mathfrak p$, with $\mathfrak k=\C\langle H\rangle$ (the complexified Lie algebra of $K$) 
and $\mathfrak p=\C\langle E,F\rangle$.

Given $\epsilon=\pm1$, let
\begin{equation}
\Z_\epsilon=\{j\in\Z\mid (-1)^j=\epsilon\}
\end{equation}
and let $V_\epsilon$ be the complex vector space with basis $\{v_j\mid j\in\Z_\epsilon\}$. 

For  $\epsilon=\pm1, \nu\in\C$ define a $(\g,K)$-module $(\pi(\epsilon,\nu),V_\epsilon)$ as follows.
Since $K$ is connected it is enough to define the representation of $\g$.
For $j\in\Z_\epsilon$ define
\begin{equation}
\label{e:action}
\begin{aligned}
\pi(\epsilon,\nu)(H)v_j&=jv_j\\
\pi(\epsilon,\nu)(E)v_j&=\tfrac12(\nu+(j+1))v_{j+2}\\
\pi(\epsilon,\nu)(F)v_j&=\tfrac12(\nu-(j-1))v_{j-2}
\end{aligned}
\end{equation}
As usual we refer to $j\in \Z_{-\epsilon}$ as a weight of $K$, and $\langle v_j\rangle$ as the $j$-weight space. 

It is straightforward to check these operators satisfy \eqref{e:defining}, 
and this defines a $(\g,K)$-module. 
The Casimir
element 
\begin{equation}
\label{e:casimir}
\Omega=H^2+2(EF+FE)=H^2-2H+4EF
\end{equation}
is in  the center of the universal enveloping algebra, 
and it acts on $\pi(\epsilon,\nu)$ by the scalar $\nu^2-1$.
The Harish-Chandra homomorphism identifies the infinitesimal character of this representation 
with $\nu\in\k^*$.

Note that
\begin{subequations}
\label{e:red}
\renewcommand{\theequation}{\theparentequation)(\alph{equation}}  
\begin{equation}
\label{e:zero}
\begin{aligned}
\pi(\epsilon,\nu)(E)v_j&=0\Leftrightarrow j=-\nu-1\\
\pi(\epsilon,\nu)(F)v_j&=0\Leftrightarrow j=\nu+1.
\end{aligned}
\end{equation}
Since $j\in\Z_\epsilon$ we see 
the operators $\pi(\epsilon,\nu)(E)$ and $\pi(\epsilon,\nu)(F)$ are 
never $0$ on $V_\epsilon$ if and only if $\nu\not\in \Z_{-\epsilon}$. 
On the other hand if 
$\nu\in\Z_{-\epsilon}$ then 
\begin{equation}
\pi(\epsilon,\nu)(E)v_{-(\nu+1)}=\pi(\epsilon,\nu)(F)v_{\nu+1}=0. 
\end{equation}
Finally by (a) 
\begin{equation}
\pi(\epsilon,\nu)(E)v_j=0=\pi(\epsilon,\nu)(F)v_{j+2}\Leftrightarrow \epsilon=-1,\nu=0,j=-1
\end{equation}
\end{subequations}

Let $F_n$ be the irreducible representation of $\g$ of dimension $n$.
For $n\ge 1$ let $D_n$ be the unique irreducible representation of
$\g$ with lowest weight $n+1$: the weights of $D_n$ are $\{n+2k\mid k\in \Z_{\ge 0}\}$.
This is a holomorphic discrete series
representation.  For $n\le 1$ let $D_n$ be the anti-holomorphic
discrete series representation with highest weight $n-1$.  Finally let
$L^{\pm}$ be the limits of discrete series: $L^+$ has lowest weight
$1$, and $L^-$ has highest weight $-1$.

Formulas \eqref{e:red}(a-c) prove:

\begin{lemma}\hfil\newline
\label{l:reducibility}
\noindent (1) $\pi(\epsilon,\nu)$ is reducible if and only if $\nu\in \Z_{-\epsilon}$.
\medskip

\noindent (2) Suppose $\pi(\epsilon,\nu)$ is reducible. Then it is completely reducible if and only if 
$(\epsilon,\nu)=(-1,0)$, in which case $\pi(-1,0)\simeq L^+\oplus L^-$.
\medskip

\noindent (3) Suppose $n\in\Z_{-\epsilon}$, $n\ne 0$. Then $\pi(\epsilon,n)$ is indecomposable. 
If $n>0$ then  $F_n$ is the unique irreducible quotient of $\pi(\epsilon,n)$, and there is an exact sequence
\begin{subequations}
\renewcommand{\theequation}{\theparentequation)(\alph{equation}}  
\label{e:exact}
\begin{equation}
0\rightarrow D_n\oplus D_{-n}\rightarrow \pi(\epsilon,n)\rightarrow F_n\rightarrow 0
\end{equation}
If $n<0$ then  $F_n$ is the unique irreducible submodule of $\pi(\epsilon,n)$, and there is an exact sequence
\begin{equation}
0\rightarrow F_n \rightarrow \pi(\epsilon,n)\rightarrow D_n\oplus D_{-n}\rightarrow 0
\end{equation}
\end{subequations}
\end{lemma}

It is convenient to change variables and write the action \eqref{e:action} as follows.
\begin{equation}
\label{e:formulas}
\begin{aligned}
\pi(\epsilon,\nu)(H)v_m&=mv_m&(m\in\Z_\epsilon)\\
\pi(\epsilon,\nu)(E)v_{m-1}&=\tfrac12(\nu+m)v_{m+1}&(m\in\Z_{-\epsilon})\\
\pi(\epsilon,\nu)(F)v_{m+1}&=\tfrac12(\nu-m)v_{m-1}&(m\in\Z_{-\epsilon})
\end{aligned}
\end{equation}

We now modify the operators to give completely reducible
representations for all $\nu$.

\begin{definition}
For $m\in\Z,\nu\in\C$ define
\begin{equation}
f_m(\nu)=
\begin{cases}
|\nu^2-m^2|^{\frac12}\frac{\nu+m}{|\nu+m|}&\nu\ne -m\\
0&\nu=-m
\end{cases}
\end{equation}
\end{definition}

\begin{proposition}
\label{p:main}
For $\epsilon=\pm1$ define a $(\g,K)$-module $(\Pi(\epsilon,\nu),V_\epsilon)$, 
with operators
$$
\begin{aligned}
\Pi(\epsilon,\nu)(H)v_m&=mv_m&(m\in \Z_\epsilon)\\
\Pi(\epsilon,\nu)(E)v_{m-1}&=f_m(\nu)v_{m+1}&(m\in \Z_{-\epsilon})\\
\Pi(\epsilon,\nu)(F)v_{m+1}&=f_{-m}(\nu)v_{m-1}&(m\in \Z_{-\epsilon})\\
\end{aligned}
$$
Then
\begin{enumerate}
\item $\Pi(\epsilon,\nu)$ is a $(\g,K)$ module;
\item $\Pi(\epsilon,\nu)\simeq \pi(\epsilon,\nu)$ for all $\nu\not\in\Z_{-\epsilon}$;
\item $\Pi(\epsilon,\nu)$ and $\pi(\epsilon,\nu)$ have the same composition factors for all $\nu$;
\item $\Pi(\epsilon,\nu)$ is completely reducible for all $\epsilon,\nu$. 
\end{enumerate}
\end{proposition}
If $\nu\in\R$,  the formulas simplify:
\begin{equation}
\label{e:real}
\begin{aligned}
\Pi(\epsilon,\nu)(H)v_m&=mv_m&(m\in\Z_\epsilon)\\
\Pi(\epsilon,\nu)(E)v_{m-1}&=\tfrac12|\nu^2-m^2|^{\frac12}\sgn(\nu+m)v_{m+1}&(m\in\Z_{-\epsilon})\\
\Pi(\epsilon,\nu)(F)v_{m+1}&=\tfrac12|\nu^2-m^2|^{\frac12}\sgn(\nu-m)v_{m-1}&(m\in\Z_{-\epsilon}).
\end{aligned}
\end{equation}
Also if $ \nu\in i\R$ we recover the original formulas \eqref{e:formulas}.

These formulas were found by a deformation process explained in Section \ref{s:deforming}.
Once they have been found, the proof of the Proposition is elementary.

\medskip

\begin{proof}
It is very easy to check the defining relations \eqref{e:defining}.   
The only nontrivial relation follows from
$$
\begin{aligned}
\pi(\epsilon,\nu)(EF)v_{m+1}&=\frac14(\nu^2-m^2)v_{m+1}\\
\pi(\epsilon,\nu)(FE)v_{m-1}&=\frac14(\nu^2-m^2)v_{m-1}
\end{aligned}
$$
which gives
$$
\begin{aligned}
\Pi(\epsilon,\nu)(EF-FE)v_{m+1}&=\frac14\{(\nu^2-m^2)-(\nu^2-(m+2)^2)\}v_{m+1}\\
&=(m+1)v_{m+1}=\Pi(\epsilon,\nu)(H)v_{m+1}
\end{aligned}
$$
The Casimir element \eqref{e:casimir} acts by the same scalar
$\nu^2-1$ in both $\pi(\epsilon,\nu)$ and $\Pi(\epsilon,\nu)$,
and these two representations have the same weights. By
\cite[Corollary 1.2.8]{vogan_green} this implies 
they are isomorphic when they are irreducible,
i.e. $\nu\not\in\Z_{-\epsilon}$. 
Also $\pi(\epsilon,\nu)$ and $\Pi(\epsilon,\nu)$ have the same composition factors, as follows 
from the facts that $\pi(\epsilon,\nu)(E)(v_j)=0\Rightarrow \Pi(\epsilon,\nu)(E)(v_j)=0)$, and similarly for $F$.
Complete reducibility is clear,
since $\Pi(\epsilon,\nu)(E)v_{m-1}=0$ if and only if
$\Pi(\epsilon,\nu)(F)v_{m+1}=0$.
\end{proof}

\begin{remark}
\label{r:not}
The operators $\Pi(\nu,\epsilon)(E)$ and $\Pi(\epsilon,\nu)(F)$ vary continuously in $\nu\in\C$. 
When $m=0$ then this dependence is algebraic:
$$
\begin{aligned}
\Pi(\epsilon,\nu)(E)(v_{-1})&=\nu v_{1}\\
\Pi(\epsilon,\nu)(F)(v_{1})&=\nu v_{-1}
\end{aligned}
$$
and $\pi(0)$ is completely reducible.
However (the matrix entries of) these  operators are not differentiable at any nonzero integer, even when restricted to $\R$. 
 these are not differentiable at any nonzero integer, even
when restricted to $\R$. 

Eyal Subag has pointed out that this phenomenon is unavoidable, and
provided the following argument. Suppose 
we are given functions $\alpha_m(\nu), \beta_m(\nu)$ defined in a neighborhood of $m$. 
Assume that $$
\begin{aligned}
\Pi(\epsilon,\nu)(H)v_m&=mv_m&(m\in\Z_\epsilon)\\
\Pi(\epsilon,\nu)(E)v_{m-1}&=\alpha_m(\nu)v_{m+1}&(m\in\Z_{-\epsilon})\\
\Pi(\epsilon,\nu)(F)v_{m+1}&=\beta_m(\nu)v_{m-1}&(m\in\Z_{-\epsilon})
\end{aligned}
$$
is a representation of $\g$, such that the Casimir element \eqref{e:casimir} acts by $\nu^2-1$. 
Then 
$$
\Pi(\epsilon,\nu)(EF)v_{m-1}=\alpha_m(\nu)\beta_m(\nu)v_{m-1}
$$
and this gives
$\pi(\epsilon,\nu)(\Omega)=(m+1)^2-2(m+1)+4\alpha_m(\nu)\beta_m(\nu)$ .
Setting this equal to $\nu^2-1$ and simplifying gives
\begin{equation}
\label{e:pole}
\alpha_m(\nu)\beta_m(\nu)=\frac14(\nu^2-m^2).
\end{equation}
The module is completely reducible at $m$ if and only if  
$\alpha_m(m)=\beta_m(m)=0$. If $\alpha,\beta$ are algebraic this gives a pole of order $2$ at $m$, whereas the right hand side has a simple pole at $m$
unless $m=0$.
\end{remark}

\sec{Deforming the family $\pi(\epsilon,\nu)$}
\label{s:deforming}
The formulas of Proposition \ref{p:main} were obtained by starting with the operators \eqref{e:formulas} and applying
the following deformation procedure. 

Fix $\epsilon$. Recall $V_\epsilon$ has basis 
$\{v_j\mid j\in\Z_\epsilon\}$. 
For each $j\in\Z_\epsilon$ choose a function 
$$
\phi_j:\C-\Z_{-\epsilon}\rightarrow \C-\{0\}.
$$
\begin{subequations}
\renewcommand{\theequation}{\theparentequation)(\alph{equation}}  
For $\nu\in\C-\Z_{-\epsilon}$ define a new basis of the representation $\pi(\epsilon,\nu)$ on $V_\epsilon$:
\begin{equation}
w_j(\nu)=\phi_j(\nu)v_j\quad(j\in\Z_\epsilon).
\end{equation}
For fixed $\nu$ the map $v_j\rightarrow w_j(\nu)$  is an isomorphism of vector
spaces. 
Set
\begin{subequations}
\renewcommand{\theequation}{\theparentequation)(\alph{equation}}  
\label{e:psi}
\begin{equation}
\label{e:psi_m}
\psi_m(\nu)=\tfrac{\phi_{m+1}(\nu)}{\phi_{m-1}(\nu)}\quad(m\in\Z_{-\epsilon},\nu\not\in\Z_{-\epsilon}).
\end{equation}
In the new basis the action \eqref{e:formulas} becomes
\end{subequations}
\begin{equation}
\begin{aligned}
\pi(\epsilon,\nu)(H)w_j(\nu)&=jw_j(\nu)\quad(j\in\Z_{\epsilon}\\
\pi(\epsilon,\nu)(E)w_{m-1}(\nu)&=\tfrac12(\nu+m)\frac1{\psi_m(\nu)}w_{m+1}(\nu)\quad(m\in\Z_{-\epsilon})\\
\pi(\epsilon,\nu)(F)w_{m+1}(\nu)&=\tfrac12(\nu-m)\psi_m(\nu)w_{m-1}(\nu)\quad(m\in\Z_{-\epsilon}).
\end{aligned}
\end{equation}
We can view this change of basis in two ways: as representing
the same linear transformation with respect to a new basis, or fixing
the basis, and conjugating the matrix to obtain a new linear
transformation. Thus for $\nu\not\in\Z_{-\epsilon}$ 
define the invertible diagonal linear transformation
$S(\nu)(v_j)=\phi_j(\nu)v_j$ $(j\in\Z_{-\epsilon})$.  
Then define
$$
\Pi(\epsilon,\nu)=S(\nu)\pi(\epsilon,\nu)S(\nu)\inv\quad(\nu\not\in\Z_{-\epsilon}).
$$
Thus $\Pi(\epsilon,\nu)\simeq\pi(\epsilon,\nu)$, and in the 
original basis $\{v_j\}$ we have (still for $\nu\not\in\Z_{-\epsilon}$):
\begin{equation}
\label{e:psiformulasfixed}
\begin{aligned}
\Pi(\epsilon,\nu)(H)v_j&=jv_j\quad(j\in \Z_\epsilon)\\
\Pi(\epsilon,\nu)(E)v_{m-1}&=\tfrac12(\nu+m)\frac1{\psi_m(\nu)}v_{m+1}\quad(m\in\Z_{-\epsilon})\\
\Pi(\epsilon,\nu)(F)v_{m+1}&=\tfrac12(\nu-m)\psi_m(\nu)v_{m-1}\quad(m\in\Z_{-\epsilon})\\
\end{aligned}
\end{equation}
\end{subequations}

What this accomplishes is that, depending on the choice of the functions $\psi_j(\nu)$,
specifically their poles and zeros, we may be able to take the limit
as $\nu$ approaches a point in $\Z_{-\epsilon}$. 
\begin{lemma}
\label{l:limit}
Suppose that 
\begin{equation}
\label{e:limits}
\lim_{\nu\rightarrow\nu_0}\frac{\nu+m}{\psi_m(\nu)}\text{ and } \lim_{\nu\rightarrow\nu_0}(\nu-m)\psi_m(\nu)
\end{equation}
exist for all $m,\nu_0\in \Z_{-\epsilon}$.
Then for all $X\in\g$, the limit
\begin{equation}
\label{e:Pi}
\Pi(\epsilon,\nu_0)(X)=\lim_{\nu\rightarrow\nu_0}S(\nu)\pi(\epsilon,\nu)(X)S(\nu)\inv
\end{equation}
exists, and gives a well defined representation 
$\Pi(\epsilon,\nu)$  for all $\nu$. 
We have:
\begin{enumerate}
\item[(1)] $\Pi(\epsilon,\nu)\simeq\pi(\epsilon,\nu)$ is irreducible for
  $\nu\not\in\Z_{-\epsilon}$;
\item[(2)] $\Pi(\epsilon,\nu)$ and $\pi(\epsilon,\nu)$ have the same
  composition factors for all $\nu$.
\end{enumerate}
\end{lemma}

The proof is essentially the same as that of Proposition \ref{t:main}.

Now we find conditions on the functions $\psi_m(\nu)$ which give  complete reducibility.
\begin{lemma}
Assume \eqref{e:limits} holds, and also 
\begin{equation}
\label{e:limit}
\lim_{\nu\rightarrow \pm
  m}\frac{\nu+m}{\psi_m(\nu)}=\lim_{\nu\rightarrow \pm
  m}(\nu-m)\psi_m(\nu)=0\quad(\text{for all }m\in\Z_{-\epsilon})
\end{equation}
Then $\Pi(\epsilon,\nu)$ is completely reducible for all $\nu$.
\end{lemma}

\begin{proof}
The conditions imply
$$
\Pi(\epsilon,\nu)(E)v_{m-1}=0\Leftrightarrow \nu=\pm m\Leftrightarrow 
\Pi(\epsilon,\nu)(F)v_{m+1}=0.
$$
\end{proof}

\begin{lemma}
\label{l:limit2}
Assume 
\begin{equation}
\label{e:psi_m2}
\psi_m(\nu)=\frac{|\nu+m|^{\frac12}}{|\nu-m|^{\frac12}}\text{ for all }\nu\not\in \Z_{-\epsilon}.
\end{equation}
Then the limits \eqref{e:limits} exist for all $m,\nu_0\in\Z_{-\epsilon}$, and \eqref{e:limit} holds.
\end{lemma}

Finally we confirm that the appropriate choice of the functions
$\phi_m(\nu)$ give $\psi_m(\nu)$ as in the Lemma. 
The $\phi_m(\nu)$ are defined recursively.
Set $\phi_0(\nu)=1$ and $\phi_1(\nu)=1$ for all $\nu$. For
$n>1$ define
\begin{subequations}
\renewcommand{\theequation}{\theparentequation)(\alph{equation}}  
\begin{equation}
\phi_n(\nu)=\frac{|\nu+n-1|^{\frac12}}{|\nu-n+1|^{\frac12}}\phi_{n-2}(\nu)  \quad(\nu\not\in\Z_{-\epsilon})
\end{equation}
and for $n\le -1$ define
\begin{equation}
\phi_{n}(\nu)=\phi_{-n}(\nu)\quad(\nu\not\in\Z_{-\epsilon}).
\end{equation}
Then \eqref{e:psi_m} holds for all $m\not\in\Z_{-\epsilon}$.
\end{subequations}
This is straightforward. Explicitly $\phi_{\pm1}(\nu)=1$ (for all $\nu$)  and for $n>1$ we have:

$$
\begin{aligned}
\phi_m(\nu)=
\begin{cases}
\frac{|\nu+1|^{\frac12}}{|\nu-1|^{\frac12}}  
\frac{|\nu+3|^{\frac12}}{|\nu-3|^{\frac12}}  
\dots
\frac{|\nu+m-1|^{\frac12}}{|\nu-m+1|^{\frac12}}  &m\text{ even}\\
\frac{|\nu+2|^{\frac12}}{|\nu-2|^{\frac12}}  
\frac{|\nu+4|^{\frac12}}{|\nu-4|^{\frac12}}  
\dots
\frac{|\nu+m-1|^{\frac12}}{|\nu-m+1|^{\frac12}}  &m\text{ odd}\\
\end{cases}\quad(\nu\not\in\Z_{-\epsilon})
\end{aligned}
$$
Plugging  $\psi_m$ of Lemma \ref{l:limit2} into \eqref{e:psiformulasfixed} gives the operators of Proposition \ref{p:main}.

\subsec{Other families}

It is natural to ask if we can deform the principal series
$\pi(\epsilon,\nu)$ differently, so that when $\nu\in\Z_{-\epsilon}$
it has the same composition factors, but with a different composition
series. The indecomposable representations were classified by Howe and
Tan \cite{howe_tan}, and we use their notation to visualize the
possiblities. 
We write the composition series of \eqref{e:exact}(a) as 
$$
(D_{-n}\,]\,F_n\,[\,D_n).
$$
The square brackets indicate that $D_{\pm n}$ are submodules:
$\pi(E)(v_{-n+1})=\pi(F)(v_{n+1})=0$. 
On the other hand $F_n$ is not a submodule, and $\pi(E)(v_{n-1})\ne 0$, $\pi(F)(v_{-n+1})\ne 0$. 
Similarly the picture for \eqref{e:exact}(b) is
$$
(D_{-n}\,[\,F_n\,]\,D_n).
$$
In the exceptional case $\pi(-1,0)$ is completely reducible:
$$
(L^-\,]\,[\,L^+)
$$

By modifying the construction we can obtain families which specialize to the indecomposable 
representations $(D_{-n}\,[\,F_n\,[\,D_n)$ and 
$(D_{-n}\,]\,F_n\,]\,D_n)$.
It turns out, unlike the family of Theorem \ref{t:main}, these families 
can be constructed algebraically.
We just state the result, the proof is just like the proof of Proposition \ref{p:main}.
These families are  also   constructed in \cite{bernstein_higson_subag} and \cite{noort}. 

\begin{lemma}
Define the representations $\pi'(\epsilon,\nu)$ by $\pi'(H)(v_j)=jv_j$ as usual, 
and
$$
\begin{aligned}
\pi'(\epsilon,\nu)(E)(v_{m-1})&=(\nu-|m|)v_{m+1}\\
\pi'(\epsilon,\nu)(F)(v_{m+1})&=(\nu+|m|)v_{m-1}\\
\end{aligned}
$$
Then (1-3) of Proposition \ref{p:main} hold. 
Suppose  $n\in \Z_{-\epsilon}$. 

\noindent (4a) If $n>0$ then $\pi'(\epsilon,n)$ has the composition series
$$
(D_{-n}\,]\,F_n\,]\,D_n)
$$

\noindent (4b) If $n<0$ then $\pi'(\epsilon,n)$ has the composition series
$$
(D_{-n}\,[\,F_n\,[\,D_n)
$$
\end{lemma}
As before $\pi'(-1,0)$ is completely reducible with factors $L^{\pm}$.

\sec{Hermitian Forms}
\label{s:hermitian}

It is well known that $\pi(\epsilon,\nu)$ admits a non-zero invariant
Hermitian form $\langle \,,\,\rangle_\nu$  if and only if 
If $\nu\in\R\cup i\R$. We do not assume this form is (positive or
negative) definite; this holds if and only if $\nu\in i\R$ or
$\nu\in\R, |\nu|< 1$.)
 If $\nu\not\in\Z_{-\epsilon}$ then (since $\pi(\epsilon,\nu)$
is irreducible) this form is non-degenerate, and unique up
to a real, non-zero scalar multiple. On the other hand if $\nu\in\Z_{-\epsilon}$, so
$\pi(\epsilon,\nu)$ is reducible, there is no non-degenerate
invariant form (except in the case $(\epsilon,\nu)=(-1,0)$): the
existence of such a form would imply $\pi(\epsilon,\nu)$ is completely reducible.
Instead, if $\pi(\epsilon,\nu_0)$ is indecomposable,
the form $\lim_{\nu\rightarrow\nu_0}\langle \,,\,\rangle_\nu$ 
is an invariant degenerate form on $\pi(\epsilon,\nu_0)$, 
whose radical is the socle.

By contrast, the family $\Pi(\epsilon,\nu)$, being completely
reducible for all $\nu$, does support a non-degenerate Hermitian form
for all $\nu$. It turns out that in this setting taking the limit, as $\nu$
approaches $\nu_0\in\Z_{-\epsilon}$, gives a nondegenerate form on
$\pi(\epsilon,\nu_0)$.  In some sense the existence of this family of
non-degenerate forms explains why the representations are completely
reducible for all $\nu$.

We now  make the preceding discussion  precise, starting with 
the definition of an invariant Hermitian form
on
a $(\g,K)$-modules $(\pi,V)$. First of all this is  a Hermitian 
form $\langle\,,\,\rangle$ on $V$. The invariance condition with respect to the action of
$\g$ is:

\begin{equation}
\label{e:EF}
\langle \pi(X)v,w\rangle+\langle v,\pi(\sigma(X))w=0\quad(v,w\in V,X\in\g)
\end{equation}
where $\sigma$ denotes complex conjugation of
$\g$ with respect to the real Lie algebra $\g_0$. 
We don't need the corresponding condition on the $K$-action. See
\cite{vogan_green}.

Now suppose $\langle\,,\,\rangle_\nu$ is an invariant Hermitian form
on $\pi=\pi(\epsilon,\nu)$.
The invariance condition applied to $H$ implies $\langle
v_i,v_j\rangle_\nu=0$ if $i\ne j$. 
Since $\sigma(E)=F$ we see 
$$
\langle \pi(E)v_{m-1},v_{m+1}\rangle_\nu + 
\langle v_{m-1},\pi(F)v_{m+1}\rangle_\nu=0
$$
Using \eqref{e:formulas} we conclude
\begin{equation}
\label{e:recursion}
(\nu+m)\langle v_{m+1},v_{m+1}\rangle_\nu
+
(\overline\nu-m)\langle v_{m-1},v_{m-1}\rangle_\nu=0
\end{equation}

Suppose $\pi$ is indecomposable.  Then there is a vector $v$
satisying $\pi(E)v\ne 0,\pi(FE)(v)=0$, or 
$\pi(F)v\ne 0,\pi(EF)v=0$.
In the first case by by \eqref{e:EF}  we have
$$
\langle \pi(E)v,\pi(E)v\rangle=-\langle v,\pi(FE)v\rangle =0
$$
This implies $\langle\,,\,\rangle$ vanishes on the irreducible summand 
containing $\pi(E)v$. The other case is similar.
This proves:

\begin{lemma}
\label{l:radical}
Suppose $\pi(\epsilon,\nu)$ is indecomposable and $\langle
\,,\,\rangle$ is an invariant Hermitian form.  Then $\langle
\,,\,\rangle$ restricted to any irreducible submodule is $0$.
\end{lemma}

On the other hand if $\pi=\pi(\epsilon,\nu)$ is irreducible 
\eqref{e:recursion} gives an inductive formula for $\langle v_{m+1},v_{m+1}\rangle_\nu$
starting with the normalization
\begin{equation}
\label{e:normalize}
\langle v_i,v_i\rangle_\nu=1\quad (i=0,1).
\end{equation}

Let us repeat this calculation with $\Pi(\epsilon,\nu)$, using the formulas from 
\eqref{e:real}. 
 If $\nu\not\in
\Z_{-\epsilon}$  we see
$$
\frac12|\nu^2-m^2|\sgn(\nu+m)\langle v_{m+1},v_{m+1}\rangle_\nu
+
\frac12|\nu^2-m^2|\sgn(\nu-m)\langle v_{m-1},v_{m-1}\rangle_\nu=0.
$$
which simplifies to
$$
\langle v_{m+1},v_{m+1}\rangle_\nu =-\sgn(\nu^2-m^2)\langle v_{m-1},v_{m-1}\rangle_\nu
$$
This proves:

\begin{lemma}
Fix $\epsilon$ and consider the family $\Pi(\epsilon,\nu)$ with
$\nu\in \R$.

Assume $\nu\not\in\Z_{-\epsilon}$. 
Let $\langle\,,\,\rangle_\nu$ be the invariant Hermitian 
form on $\Pi(\epsilon,\nu)$ normalized 
as in \eqref{e:normalize}. 
Then 
\begin{equation}
\langle v_j,v_j\rangle_\nu=\pm1\quad\text{for all }j.
\end{equation}
\end{lemma}

Note that the dependence of $\langle\,,\,\rangle_\nu$ on $\nu$ is very
weak: the signs $\langle v_j,v_j\rangle_\nu$ alternate for
$|j|<|\nu|-1$, 
and then are constant for $|j|>|\nu|-1$.

If $\nu_0\in\Z_{-\epsilon}$, it is clear that both left and right
limits $\lim_{\nu\rightarrow \nu_0^\pm}\langle v,v\rangle _{\nu}$
exist for all $v$. Also   for all  $j$ 
$\lim_{\nu\rightarrow \nu_0^\pm}|\langle v_j,v_j)\rangle _{\nu}|=\pm1$, and 
$$
\lim_{\nu\rightarrow \nu_0^+}\langle v_j,v_j\rangle _{\nu}=\pm
\lim_{\nu\rightarrow \nu_0^-}\langle v_j,v_j\rangle _{\nu},
$$
This proves:

\begin{proposition}
\label{p:hermitian}
The family $\Pi(\epsilon,\nu)\quad(\nu\in\R)$ admits a  family of non-degenerate Hermitian
forms $\langle\,,\,\rangle_\nu$, 
satisfying 
$\langle v_j,v_j\rangle_\nu=\pm1$ for all $j$. 
The family of forms can be chosen to be upper or lower-semicontinuous 
at each $\nu\in\Z_{-\epsilon}$. 
\end{proposition}

Return for the moment to the family of principal series
$\pi(\epsilon,\nu)$. For simplicity suppose $\epsilon=1$ and $n\ge 2$
is even. The family of forms $\langle \,,\,\rangle_\nu$ are
nondegenerate in a punctured neighborhood of $n$.  Furthermore
$\lim_{\nu\rightarrow n}\langle \,,\,\rangle_\nu= \langle
\,,\,\rangle_n$,
and the radical of this form is the sum of the two discrete series
representations. By taking residues, Jantzen defines an invariant
Hermitian form on the associated graded representation, i.e. on the
finite dimensional composition factor \cite{jantzen}, see  \cite[Definition
14.8]{unitaryDual}.
 This form plays an important
role in the unitarity algorithm of \cite{unitaryDual}.

We see that the family $\Pi(\epsilon,\nu)$ provides an alternative
description of the Jantzen form. In this setting we obtain a
family of nondegenerate Hermitian forms
$\langle \,,\,\rangle_\nu$ on the  irreducible representations
$\Pi(\epsilon,\nu)$ ($\nu\not\in\Z_{-\epsilon}$). Taking the limit 
as $\nu$ approaches a reducibility point $\nu_0$
gives a nondegenerate form on the completely reducible representation
$\Pi(\epsilon,\nu_0)$, which agrees (passing to the associated graded
representation) with the Jantzen form.

\bibliographystyle{plain}
%\bibliography{/home/jda/bibliographies/refs}

\begin{thebibliography}{1}

\bibitem{unitaryDual}
J.~Adams, Peter Trapa, Marc van Leeuwen, and David A.~Jr. Vogan.
\newblock Unitary dual of real reductive groups.
\newblock preprint, arXiv:1212.2192.

\bibitem{howe_tan}
Roger Howe and Eng-Chye Tan.
\newblock {\em Nonabelian harmonic analysis}.
\newblock Universitext. Springer-Verlag, New York, 1992.
\newblock Applications of ${{\rm{S}}L}(2,{{\bf{R}}})$.

\bibitem{jantzen}
Jens~Carsten Jantzen.
\newblock {\em Moduln mit einem h\"ochsten {G}ewicht}, volume 750 of {\em
  Lecture Notes in Mathematics}.
\newblock Springer, Berlin, 1979.

\bibitem{bernstein_higson_subag}
Nigel~Higson Joseph~Bernstein and Eyal Subag.
\newblock Algebraic families of harish-chandra pairs.
\newblock preprint, arXiv:1306.2729.

\bibitem{naimark}
M.~A. Naimark.
\newblock {\em Linear representations of the {L}orentz group}.
\newblock Translated by Ann Swinfen and O. J. Marstrand; translation edited by
  H. K. Farahat. A Pergamon Press Book. The Macmillan Co., New York, 1964.

\bibitem{noort}
Vincent van~der Noort.
\newblock Analytic parameter dependence of harish-chandra modules for real
  reductive lie groups - a family affair.
\newblock thesis, University of Utrecht.

\bibitem{vogan_green}
David~A. Vogan, Jr.
\newblock {\em Representations of real reductive {L}ie groups}, volume~15 of
  {\em Progress in Mathematics}.
\newblock Birkh\"auser Boston, Mass., 1981.

\end{thebibliography}
\def\cprime{$'$} \def\cftil#1{\ifmmode\setbox7\hbox{$\accent"5E#1$}\else
  \setbox7\hbox{\accent"5E#1}\penalty 10000\relax\fi\raise 1\ht7
  \hbox{\lower1.15ex\hbox to 1\wd7{\hss\accent"7E\hss}}\penalty 10000
  \hskip-1\wd7\penalty 10000\box7}
  \def\cftil#1{\ifmmode\setbox7\hbox{$\accent"5E#1$}\else
  \setbox7\hbox{\accent"5E#1}\penalty 10000\relax\fi\raise 1\ht7
  \hbox{\lower1.15ex\hbox to 1\wd7{\hss\accent"7E\hss}}\penalty 10000
  \hskip-1\wd7\penalty 10000\box7}
  \def\cftil#1{\ifmmode\setbox7\hbox{$\accent"5E#1$}\else
  \setbox7\hbox{\accent"5E#1}\penalty 10000\relax\fi\raise 1\ht7
  \hbox{\lower1.15ex\hbox to 1\wd7{\hss\accent"7E\hss}}\penalty 10000
  \hskip-1\wd7\penalty 10000\box7}
  \def\cftil#1{\ifmmode\setbox7\hbox{$\accent"5E#1$}\else
  \setbox7\hbox{\accent"5E#1}\penalty 10000\relax\fi\raise 1\ht7
  \hbox{\lower1.15ex\hbox to 1\wd7{\hss\accent"7E\hss}}\penalty 10000
  \hskip-1\wd7\penalty 10000\box7} \def\cprime{$'$} \def\cprime{$'$}
  \def\cprime{$'$} \def\cprime{$'$} \def\cprime{$'$} \def\cprime{$'$}
  \def\cprime{$'$} \def\cprime{$'$}

\enddocument
\end